\documentclass[12pt]{article}
\usepackage{amssymb, amscd, amsthm, amsmath,latexsym, amstext,mathrsfs,verbatim,longtable,array,multirow}
\usepackage[all]{xy}
\usepackage{graphicx,enumerate, url}

\def\lijntje{\vrule height2.4pt depth-2pt width0.5in}

\def\vlijntje{\vrule height0.45in depth0.4pt width0.4pt}
\def\vlijn{\buildrel {\hbox to 0pt{\hss$\textstyle\circ$\hss}}\over\vlijntje}

\def\dlijntje{{\vrule height2pt depth-1.6pt
width0.5in}\llap{\vrule height4pt depth-3.6pt width0.5in}}
\def\tlijntje{{\vrule height1.7pt depth-1.3pt
width0.5in}\llap{\vrule height3.0pt depth-2.6pt width0.5in}\llap{\vrule height4.3pt depth-3.9pt width0.5in}
}
\def\vtriple#1\over#2\over#3{\mathrel{\mathop{\kern0pt #2}\limits_{\hbox
to 0pt{\hss$#1$\hss}}^{\hbox to 0pt{\hss$#3$\hss}}}}
\def\rvtriple#1\over#2\over#3{\mathrel{\mathop{\kern0pt #2}\limits_{\hbox
to 0pt{\hss$#3$\hss}}^{\hbox to 0pt{\hss$#1$\hss}}}}
\def\Ct{\vtriple{\scriptstyle 2}\over\circ\over{}
\kern-1pt\lijntje\kern-1pt\vtriple{\scriptstyle 1}\over\circ\over{}
\kern-4pt{\dlijntje \kern -25pt<}\kern8pt
\vtriple{\scriptstyle 0}\over\circ\over{}\kern-1pt
}
\def\Bt{\vtriple{\scriptstyle 2}\over\circ\over{}
\kern-1pt\lijntje\kern-1pt\vtriple{\scriptstyle 1}\over\circ\over{}
\kern-4pt{\dlijntje \kern -25pt>}\kern8pt
\vtriple{\scriptstyle 0}\over\circ\over{}\kern-1pt}

\def\ddA{{\rm A}}

\def\ddC{{\rm C}}

\newcommand{\C}{\mathbb C}

\newcommand{\R}{\mathbb R}
\newcommand{\Z}{\mathbb Z}

\def\Dm{\vtriple{\scriptstyle n+1}\over\circ\over{}\kern-1pt\lijntje\kern-1pt
\vtriple{\scriptstyle{n}}\over\circ\over{}
\cdots\cdots\vtriple{\scriptstyle 4}\over\circ\over{}\kern-1pt\lijntje\kern-1pt
\vtriple{\scriptstyle 3}\over\circ\over{\buildrel
{\scriptstyle 2}\over\vlijn}\kern-1pt\lijntje\kern-1pt
\vtriple{1}\over\circ\over{}\kern-1pt}

\def\Dn{\vtriple{\scriptstyle n}\over\circ\over{}\kern-1pt\lijntje\kern-1pt
\vtriple{\scriptstyle{n-1}}\over\circ\over{}
\cdots\cdots\vtriple{\scriptstyle 4}\over\circ\over{}\kern-1pt\lijntje\kern-1pt
\vtriple{\scriptstyle 3}\over\circ\over{\buildrel
{\scriptstyle 2}\over\vlijn}\kern-1pt\lijntje\kern-1pt
\vtriple{1}\over\circ\over{}\kern-1pt}
\def\En{\vtriple{\scriptstyle n}\over\circ\over{}\kern-1pt\lijntje\kern-1pt
\vtriple{\scriptstyle{n-1}}\over\circ\over{}
\cdots\cdots\vtriple{\scriptstyle 5}\over\circ\over{}\kern-1pt\lijntje\kern-1pt
\vtriple{\scriptstyle 4}\over\circ\over{\buildrel
{\scriptstyle 2}\over\vlijn}\kern-1pt\lijntje\kern-1pt
\vtriple{\scriptstyle 3}\over\circ\over{}\kern-1pt\lijntje\kern-1pt
\vtriple{\scriptstyle 1}\over\circ\over{}\kern-1pt}
\def\An{\vtriple{\scriptstyle n}\over\circ\over{}\kern-1pt\lijntje\kern-1pt
\vtriple{\scriptstyle{n-1}}\over\circ\over{}\kern-1pt\lijntje\kern-1pt
\vtriple{\scriptstyle n-2}\over\circ\over{}
\cdots\cdots
\vtriple{\scriptstyle 2}\over\circ\over{}\kern-1pt\lijntje\kern-1pt
\vtriple{\scriptstyle 1}\over\circ\over{}\kern-1pt}
\def\Cn{\vtriple{\scriptstyle n-1}\over\circ\over{}
\kern-1pt\lijntje\kern-1pt\vtriple{\scriptstyle{n-2}}\over\circ\over{}
\cdots\cdots
\vtriple{\scriptstyle 2}\over\circ\over{}
\kern-1pt\lijntje\kern-1pt\vtriple{\scriptstyle 1}\over\circ\over{}
\kern-4pt{\dlijntje \kern -25pt<}\kern10pt
\vtriple{\scriptstyle 0}\over\circ\over{}\kern-1pt}
\def\Ct{\vtriple{\scriptstyle 2}\over\circ\over{}
\kern-1pt\lijntje\kern-1pt\vtriple{\scriptstyle 1}\over\circ\over{}
\kern-4pt{\dlijntje \kern -25pt<}\kern12pt
\vtriple{\scriptstyle 0}\over\circ\over{}\kern-1pt
}
\def\Bn{\vtriple{\scriptstyle n-1}\over\circ\over{}
\kern-1pt\lijntje\kern-1pt\vtriple{\scriptstyle{n-2}}\over\circ\over{}
\cdots\cdots
\vtriple{\scriptstyle 2}\over\circ\over{}
\kern-1pt\lijntje\kern-1pt\vtriple{\scriptstyle 1}\over\circ\over{}
\kern-4pt{\dlijntje \kern -25pt>}\kern10pt
\vtriple{\scriptstyle 0}\over\circ\over{}\kern-1pt}
\def\Bt{\vtriple{\scriptstyle 2}\over\circ\over{}
\kern-1pt\lijntje\kern-1pt\vtriple{\scriptstyle 1}\over\circ\over{}
\kern-4pt{\dlijntje \kern -25pt>}\kern12pt
\vtriple{\scriptstyle 0}\over\circ\over{}\kern-1pt}
\def\Es{\vtriple{\scriptstyle 6}\over\circ\over{}\kern-1pt\lijntje\kern-1pt
\vtriple{\scriptstyle 5}\over\circ\over{}\kern-1pt\lijntje\kern-1pt
\vtriple{\scriptstyle 4}\over\circ\over{\buildrel
{\scriptstyle 2}\over\vlijn}\kern-1pt\lijntje\kern-1pt
\vtriple{3}\over\circ\over{}\kern-1pt\lijntje\kern-1pt
\vtriple{\scriptstyle 1}\over\circ\over{}\kern-1pt}
\def\Ff{
\vtriple{\scriptstyle 1}\over\circ\over{}
\kern-1pt\lijntje\kern-1pt\vtriple{\scriptstyle 2}\over\circ\over{}
\kern-4pt{\dlijntje \kern -25pt<}\kern10pt
\vtriple{\scriptstyle 3}\over\circ\over{}\kern-1pt\lijntje\kern-1pt
\vtriple{\scriptstyle 4}\over\circ\over{}
\kern-1pt}
\def\Ht{
\vtriple{\scriptstyle 1}\over\circ\over{}
\kern-1pt\overset{5}{\lijntje}\kern-1pt\vtriple{\scriptstyle 2}\over\circ\over{}
\kern-1pt\lijntje\kern-1pt
\vtriple{\scriptstyle 3}\over\circ\over{}\kern-1pt}
\def\Hf{
\vtriple{\scriptstyle 1}\over\circ\over{}
\kern-1pt\overset{5}{\lijntje}\kern-1pt\vtriple{\scriptstyle 2}\over\circ\over{}
\kern-1pt\lijntje\kern-1pt
\vtriple{\scriptstyle 3}\over\circ\over{}\kern-1pt\lijntje\kern-1pt
\vtriple{\scriptstyle 4}\over\circ\over{}
\kern-1pt}
\def\In{
\vtriple{\scriptstyle 0}\over\circ\over{}
\kern-1pt\overset{n}{\lijntje}\kern-1pt\vtriple{\scriptstyle 1}\over\circ\over{}
\kern-1pt}
\def\Gt{
\vtriple{\scriptstyle 0}\over\circ\over{}
\kern-4pt{\tlijntje\kern -25pt<}\kern 10pt\vtriple{\scriptstyle 1}\over\circ\over{}
\kern-1pt}
\def\EBn{\vtriple{\scriptstyle n-1}\over\circ\over{}
\kern-1pt\lijntje\kern-1pt\vtriple{\scriptstyle{n-2}}\over\circ\over{\buildrel
{\scriptstyle -1}\over\vlijn}\cdots\cdots
\vtriple{\scriptstyle 2}\over\circ\over{}
\kern-1pt\lijntje\kern-1pt\vtriple{\scriptstyle 1}\over\circ\over{}
\kern-4pt{\dlijntje \kern -25pt<}\kern8pt
\vtriple{\scriptstyle 0}\over\circ\over{}\kern-1pt}
\def\Cn{\vtriple{\scriptstyle n-1}\over\circ\over{}
\kern-1pt\lijntje\kern-1pt\vtriple{\scriptstyle{n-2}}\over\circ\over{}
\cdots\cdots
\vtriple{\scriptstyle 2}\over\circ\over{}
\kern-1pt\lijntje\kern-1pt\vtriple{\scriptstyle 1}\over\circ\over{}
\kern-4pt{\dlijntje \kern -25pt<}\kern10pt
\vtriple{\scriptstyle 0}\over\circ\over{}\kern-1pt}
\def\ECn{\vtriple{\scriptstyle -2}\over\circ\over{}
\kern-4pt{\dlijntje \kern -25pt>}\kern8pt\vtriple{\scriptstyle n-1}\over\circ\over{}
\kern-1pt\lijntje\kern-1pt\vtriple{\scriptstyle{n-2}}\over\circ\over{}
\cdots\cdots
\vtriple{\scriptstyle 2}\over\circ\over{}
\kern-1pt\lijntje\kern-1pt\vtriple{\scriptstyle 1}\over\circ\over{}
\kern-4pt{\dlijntje \kern -25pt<}\kern12pt
\vtriple{\scriptstyle 0}\over\circ\over{}\kern-1pt}
\def\Fo{\vtriple{\scriptstyle -1}\over\circ\over{}
\kern-1pt\lijntje\kern-1pt
\vtriple{\scriptstyle 1}\over\circ\over{}
\kern-1pt\lijntje\kern-1pt\vtriple{\scriptstyle 2}\over\circ\over{}
\kern-4pt{\dlijntje \kern -25pt<}\kern8pt
\vtriple{\scriptstyle 3}\over\circ\over{}\kern-1pt\lijntje\kern-1pt
\vtriple{\scriptstyle 4}\over\circ\over{}
\kern-1pt}
\def\Ft{
\vtriple{\scriptstyle 1}\over\circ\over{}
\kern-1pt\lijntje\kern-1pt\vtriple{\scriptstyle 2}\over\circ\over{}
\kern-4pt{\dlijntje \kern -25pt<}\kern8pt
\vtriple{\scriptstyle 3}\over\circ\over{}\kern-1pt\lijntje\kern-1pt
\vtriple{\scriptstyle 4}\over\circ\over{}
\kern-1pt\lijntje\kern-1pt
\vtriple{\scriptstyle -2}\over\circ\over{}
\kern-1pt}
\def\Go{\vtriple{\scriptstyle -1}\over\circ\over{}
\kern-1pt\lijntje\kern-1pt
\vtriple{\scriptstyle 0}\over\circ\over{}
\kern-4pt{\tlijntje\kern -25pt<}\kern 12pt\vtriple{\scriptstyle 1}\over\circ\over{}
\kern-1pt}
\def\Gf{
\vtriple{\scriptstyle 0}\over\circ\over{}
\kern-4pt{\tlijntje\kern -25pt<}\kern 12pt\vtriple{\scriptstyle 1}\over\circ\over{}
\kern-1pt\lijntje\kern-1pt
\vtriple{\scriptstyle -2}\over\circ\over{}
\kern-1pt}

\numberwithin{equation}{section}

\newtheorem{lemma}{Lemma}[section]
\newtheorem{cor}[lemma]{Corollary}
\newtheorem{prop}[lemma]{Proposition}
\newtheorem{thm}[lemma]{Theorem}

\theoremstyle{remark}
\newtheorem{rem}[lemma]{Remark}

\theoremstyle{definition}

\newtheorem{defn}[lemma]{Definition}

\def\b{\beta}

\begin{document}
\title{Characteristic polynomials and finitely dimensional representations of simple Lie Algebras}
\author{Amin Geng, Shoumin Liu\footnote{The author is  funded by the NSFC (Grant No. 11971181, Grant No.11871308) },   Xumin Wang}
\date{}
\maketitle
%\mainmatter
%\setcounter{section}{-1} \tableofcontents

\begin{abstract}
In this paper, the correspondence between the finite dimensional representations of a simple Lie algebra and their characteristic polynomials is established, and a monoid structure on these characteristic polynomials is constructed.
 Furthermore, the characteristic polynomials  of $\mathfrak{sl}(2, \C )$ on some classical simple Lie algebras through adjoint representations are studied,
 and we present some results of Borel subalgebras and parabolic subalgebras of simple Lie algebras through characteristic polynomials.
\end{abstract}

\section{Introduction}
\hspace{1.3em}The determinants and eigenvalues of square matrices are most classical topics in linear algebra, and many mathematicians have defined a lot of generalizations of them and studied many related topics.
Dedekind and Frobenius started to investigate characteristic polynomial
$$f_\phi(z_0,z_1,\cdots,z_n)=\mathrm{det}(z_0I+z_1\phi(g_i)+\cdots+z_n\phi(g_n)),$$
with $\phi$ being a representation of a finite group $G=\{g_i\}_{i=1}^n$ in \cite{D1969} and \cite{F1896}.
In \cite{Y2009}, the notion of projective spectrum of operators is defined by R. Yang through the multiparameter pencil, and the multivariable homogeneous characteristic polynomial has been studied.
Many fruitful results have been obtained in \cite{CMK2016}, \cite{GY2017}, \cite{GR2014}, \cite{HY2018} and \cite{KY2021}.  It is natural to consider similar topics for Lie algebras of finite dimension.
For a Lie algebra $\mathfrak{g}$ with a basis
$\{x_1,\cdots,x_n\}$, the characteristic polynomial of its adjoint representation
$$f_{\mathfrak{g}}=\mathrm{det}(z_0I + z_{1}\mathrm{ad}x_1 + \cdots + z_{n}\mathrm{ad}x_n)$$ is investigated in \cite{KY2021}. In \cite{KY2021}, it is proved that $f_\mathfrak{g}$ is invariant under the automorphism group $Aut(\mathfrak{g})$.
Let $\phi:\mathfrak{sl}(2, \C)\rightarrow \mathfrak{gl}(V)$ be an irreducible representation of  $\mathfrak{sl}(2, \C)$, where $V$ is a $(m+1)$-dimensional complex vector space, the characteristic polynomial
\begin{eqnarray}\label{fpi}
 f_\phi(z_0,z_1,z_2,z_3)=\begin{cases}
 z_0\prod_{l=1}^{m/2}\left(z_0^2-4l^2(z_1^2+z_2z_3)\right) \quad 2\mid m.\\
 \\
 \prod_{l=0}^{(m-1)/2}\left(z_0^2-(2l+1)^2(z_1^2+z_2z_3)\right)\quad 2\nmid m.
 \end{cases}
\end{eqnarray}
has been proved in \cite{CCD2019}, \cite{HZ2019} and \cite{H2021}. When  the homomorphism $\phi$ is an arbitrary finitely dimensional representation of  $\mathfrak{sl}(2, \C)$, for  each $n\in \Z$,
let $d_{n,\phi}$  denote the dimension of  eigenvector space of $\phi(h)$ for the eigenvalue $n$,  and the characteristic polynomial
\begin{equation}\label{FIRST}
f_{\phi}(z_0,z_1,z_2,z_3)=z_0^{d_{0,\phi}}\prod_{n\geq 1}\left(z_0^2-n^2(z_1^2+z_2z_3)\right)^{d_{n,\phi}},
\end{equation}
for $\mathfrak{sl}(2, \C )$ on its representation $\phi$ is obtained in \cite{JL2022},
where the authors prove there is one to one correspondence between finite dimensional representations of $\mathfrak{sl}(2, \C )$ and their characteristic polynomials.
It is natural to consider whether we can obtain similar results for simple Lie algebra as  in \cite{JL2022}.

The paper is sketched as the following.
In Section $2$, we describe the difference of characteristic polynomials for Lie algebras under base changes.
In Section $3$, it is showed that a finite dimensional representation of a simple Lie algebra can be reconstructed through its characteristic polynomials.
In Section $4$, similar to the Lie algebra $\mathfrak{sl}(2, \C )$, characteristic polynomials of a simple Lie algebra can have a commutative monoid structure compatible with the tensor product of representations.
In Section $5$, some results about characteristic polynomials of $\mathfrak{sl}(2, \C )$ acting on the classical simple Lie algebras can be obtained by computation.
In Section $6$, we calculate the rank of the spectral matrices for the Borel subalgebras of  simple Lie algebras.

\section{Base changes and automorphisms on simple Lie algebras}

\hspace{1.3em}In this section, we will focus on how the characteristic polynomials of a complex simple Lie algebra changes with base changes. In this paper,
we assume that the readers are familiar with the basic knowledge of Lie  algebra from \cite{FH1991} or \cite{H1972}.
 Let $\mathfrak{g}$ be a $m$-dimensional Lie algebra with a basis $\mathcal{B}=\{\upsilon_1,\cdots,\upsilon_m\}$,  and $\phi:\mathfrak{g}\rightarrow \mathfrak{gl}(V)$ be a finite dimensional linear representation of $\mathfrak{g}$,
  with $V$ being a $\mathfrak{g}$-module. The characteristic polynomial with the basis $\mathcal{B}$ for the representation $\phi$  is defined as the follows.
\begin{defn}
 The polynomial $$f_{\phi,\mathcal{B}}(z_0,z_1,\cdots,z_m)=det(z_{0}I+\phi(\upsilon_1)z_1+\cdots+\phi(\upsilon_m)z_m)$$ is called the characteristic polynomial  for the representation $\phi$ under the basis $\mathcal{B}$ of $\mathfrak{g}$,
 where $z_{0}$, $z_{1}$, $\cdots$, $z_{m}$ are indeterminants.
\end{defn}
 For  convenience,  the polynomial $f_{\phi,\mathcal{B}}(z_0,z_1,\cdots,z_m)$ is denoted as $f_{\phi,\mathcal{B}}(z)$ with $z=(z_0,z_1,\cdots,z_m)$.

Suppose there are  two  bases  $$\mathcal{B}_1=\{\upsilon_1,\cdots,\upsilon_m\},\quad \quad \mathcal{B}_2=\{\omega_1,\cdots,\omega_m\}$$ for $\mathfrak{g}$.
%Let's find the relation between $f_{\phi,\mathcal{B}_1}(z)$ and $f_{\phi,\mathcal{B}_2}(z)$.
%We define a linear map $\xi$ such that $\xi(\upsilon_i)= \omega_i$ for $i=1,\cdots,m$. And we have
%$$\xi: \{\upsilon_1,\cdots,\upsilon_m\}\rightarrow \{\omega_1,\cdots,\omega_m\}$$
Suppose $\omega_i=\sum_{j=1}^{m}b_{ij}\upsilon_j$, for $i=1,\cdots, m$,  then
\begin{equation} \label{tran}
\left (\begin{array}{cccc} \omega_1\\ \omega_2\\ \vdots \\ \omega_m\end{array}\right)=B\left(\begin{array}{cccc}\upsilon_1\\ \upsilon_2 \\ \vdots \\ \upsilon_2\end{array}\right),
\end{equation}
 where the invertible matrix $B$=$(b_{ij})_{m\times m}$ is the transformation matrix from $\mathcal{B}_1$ to $\mathcal{B}_2$.  Through (\ref{tran}), for each representation $\phi$, it can be verified that
\begin{equation}
\left (\begin{array}{cccc} \phi(\omega_1)\\ \phi(\omega_2)\\ \vdots \\ \phi(\omega_m)\end{array}\right)=\left(\begin{array}{cccc} b_{11}  & b_{12}  & \cdots &b_{1m}\\b_{21}  & b_{22}  & \cdots &b_{2m}\\ \vdots &\vdots&\ddots&\vdots \\ b_{m1}  & b_{m2} &\cdots&b_{mm}\end{array}\right ) \left(\begin{array}{cccc}\phi(\upsilon_1)\\ \phi(\upsilon_2) \\ \vdots \\\phi(\upsilon_m)\end{array}\right).
\end{equation}
Let
$$D=\left(
         \begin{array}{cc}
           1 & 0 \\
           0 & B \\
         \end{array}
       \right).
$$
It follows that

\begin{eqnarray*}
f_{\phi,\mathcal{B}_2}(z)&=&\mathrm{det}\left(z_{0}I+\phi(\omega_1)z_{1}+\cdots+\phi(\omega_m)z_{m}\right)\\
                         &=&\mathrm{det}\left(\left (\begin{array}{cccc} z_{0}&z_{1}& \cdots &z_{m}\end{array}\right )\left(\begin{array}{cccc}I\\ \phi(\omega_1)\\ \vdots\\ \phi(\omega_m)\end{array}
\right)\right)\\
                         &=&\mathrm{det}\left(\left (\begin{array}{cccc}  z_{0}&z_{1}& \cdots &z_{m}\end{array}\right )\left(\begin{array}{ccccc}1 & 0 & \cdots & 0\\0& b_{11}  &\cdots &b_{1m}\\ \vdots &\vdots&\ddots&\vdots \\0& b_{m1} &\cdots&b_{mm}\end{array}\right )\left(\begin{array}{cccc}I\\ \phi(\upsilon_1)\\ \vdots\\ \phi(\upsilon_m)\end{array}
\right)\right)\\
                         &=&f_{\phi,\mathcal{B}_1}(zD).
 \end{eqnarray*}

By the argument, the theorem below can hold.
\begin{thm}\label{equi}
Let $\mathfrak{g}$ be a  Lie algebra of finite dimension with any two  bases $\{\upsilon_1,\cdots,\upsilon_m\}$ and $\{\omega_1,\cdots,\omega_m\}$, there is a $B=(b_{ij})_{m\times m}\in GL_m$, such that
$$\left (\begin{array}{cccc} \omega_1\\ \omega_2\\ \vdots \\ \omega_m\end{array}\right)=B\left(\begin{array}{cccc}\upsilon_1\\ \upsilon_2 \\ \vdots \\ \upsilon_2\end{array}\right)$$
and let $$D=\left(
         \begin{array}{cc}
           1 & 0 \\
           0 & B \\
         \end{array}
       \right).
$$ Then we have
$$f_{\phi,\mathcal{B}_2}(z)=f_{\phi,\mathcal{B}_1}(zD).$$
\end{thm}
Let $\sigma(f_{\phi,\mathcal{B}_1})=\{z\ | f_{\phi,\mathcal{B}_1}(z)=0, z \in\C^{m+1}\}$ be the projective spectrum of $(I,\phi(\upsilon_1),\cdots,\phi(\upsilon_m))$ as in \cite{Y2009}. By Theorem \ref{equi},  the following corollary holds.
\begin{cor}
Let $f_{\phi,\mathcal{B}_1}(z)$, $f_{\phi,\mathcal{B}_2}(z)$, $D$ be as above, then $\sigma(f_{\phi,\mathcal{B}_2}(z))=\sigma(f_{\phi,\mathcal{B}_1}(z))D$. Therefore the spectrums defined by $f_{\phi,\mathcal{B}_1}(z)$ and $f_{\phi,\mathcal{B}_2}(z)$ are homeomorphic under a linear transformation.
\end{cor}

Let $\mathfrak{g}$ be a finite dimensional complex simple Lie algebra, $\Phi$ be the root system of $\mathfrak{g}$, and  $\Pi$=\{$\alpha_1$, $\cdots$,$\alpha_n$\} be simple roots of $\Phi$. Let ${\mathfrak{h}}$ be a Cartan subalgebra of $\mathfrak{g}$, $\{h_{\alpha_1},\cdots,h_{\alpha_n}\}$ be a basis of ${\mathfrak{h}}$ corresponding to $\Pi$, and $E_{\alpha}$ be the root vector for  each $\alpha\in\Phi$. It is well known that  the set $\mathcal{A}$=$\{h_{\alpha_1},\cdots,h_{\alpha_n}, E_{\alpha}(\alpha\in\Phi)\}$ is a canonical basis of $\mathfrak{g}$, and we have the
Cartan decomposition $$\mathfrak{g}={\mathfrak{h}}\oplus \sum_{\alpha\in\Phi} \C E_{\alpha}.$$

 Because the representations of  finitely  dimensional complex simple Lie algebras have their own special properties, we present a new proof of \cite[Theorem 2.3]{KY2021} for $\mathfrak{g}$ of this kind as the below.
\begin{thm}\label{auto}
Let $\mathfrak{g}$ be a finite dimensional complex simple Lie algebra, $\phi$ be a representation of $\mathfrak{g}$, $\mathcal{B}$ be an arbitrary basis of $\mathfrak{g}$,  and $\mathrm{Aut}(\mathfrak{g})$ be the automorphism group of $\mathfrak{g}$. For any $\tau\in \mathrm{Aut}(\mathfrak{g})$, we have $f_{\phi,\mathcal{B}}(z)=f_{\phi,\tau\mathcal{B}}(z).$
\end{thm}
\begin{proof}
 Let $\mathcal{A}$=$\{h_{\alpha_1},\cdots,h_{\alpha_n}, E_{\alpha}(\alpha\in\Phi)\}$ be a canonical basis of $\mathfrak{g}$. Under the automorphism $\tau$, the set $\tau \mathcal{A}$ is another basis of  $\mathfrak{g}$.
It is known that $\tau \{h_{\alpha_1},\cdots,h_{\alpha_n}\}$ is a basis $\tau \mathfrak{h}$,
which is also a Cartan subalgebra of $\mathfrak{g}$, and $\{\tau h_{\alpha_1},\cdots,\tau h_{\alpha_n}, \tau E_{\alpha}(\alpha\in\Phi)\}$ forms another canonical basis of $\mathfrak{g}$ such that
$\mathfrak{g}$ has a new Cartan decomposition $$\mathfrak{g}={\tau\mathfrak{h}}\oplus \sum_{\alpha\in\Phi} \C \tau E_{\alpha}.$$
It is known that the representations of a simple Lie algebra are independent of its Cartan decompositions.
Up to some base change of the $\mathfrak{g}$-module $V$, the linear maps $\phi(h_{\alpha_1})$s and $\phi(\tau h_{\alpha_1})$s, $\phi(E_\alpha)$s and $\phi(\tau E_\alpha)$s have the same matrices, respectively.
 Therefore it follows that
$$f_{\phi,\mathcal{A}}(z)=f_{\phi,\tau\mathcal{A}}(z).$$
Suppose the transformation of $\mathcal{A}$ to $\mathcal{B}$ is the matrix $B$, hence the transformation of $\tau\mathcal{A}$ to $\tau\mathcal{B}$ is also $B$. Let
$$D=\left(
         \begin{array}{cc}
           1 & 0 \\
           0 & B \\
         \end{array}
       \right).
$$
By Theorem \ref{equi}, we see that $$f_{\phi,\mathcal{B}}(z)=f_{\phi,\mathcal{A}}(zD)=f_{\phi,\tau\mathcal{A}}(zD)=f_{\phi,\tau\mathcal{B}}(z).$$
\end{proof}
By the basic knowledge of Lie groups and Weyl groups, the corollary below follows, which indicates that the characteristic polynomials are related to
the invariant theory of Lie groups and Weyl groups.
\begin{cor} \label{invar}
Let $\mathfrak{g}$, $\phi$, $\mathcal{B}$, $\Phi$, $E_\alpha$ be as above.\\
(i) Suppose $G$ is a  classical Lie group with complex simple Lie algebra $\mathfrak{g}$. For any $x\in G$, let $\mathrm{Ad}x$ be the Lie algebra automorphism on $\mathfrak{g}$
induced by the the conjugation action of $x$ on $G$.
%and $g\in G$,
%we have  $\tau_x: \mathfrak{g}\rightarrow\mathfrak{g}$, $\tau_x(g)=xgx^{-1}$ is an automorphism of $\mathfrak{g}$.
Then $$f_{\phi,\mathcal{B}}(z)=f_{\phi,\mathrm{Ad}x\mathcal{B}}(z).$$
(ii) Let $W$ be the Weyl group corresponding to the root system $\Phi$. For any $a\in W$, it has a natural action on $\Phi$,
also on $\{E_\alpha\}_{\alpha\in\Phi}$ by defining $a(E_\alpha)=E_{a\alpha}$, which can be extended to the whole Lie algebra $\mathfrak{g}$,
we denote it by $\tau_a$. Then we have $$f_{\phi,\mathcal{B}}(z)=f_{\phi,\tau_a\mathcal{B}}(z).$$
\end{cor}

By Theorem \ref{equi} and Theorem \ref{auto}, we can just focus on $f_{\phi,\mathcal{A}}(z)$, where $\mathcal{A}$ is a canonical basis,
so we abbreviate $f_{\phi,\mathcal{A}}(z)=f_{\phi}(z).$ In the later sections, we will keep this notation.

\section{Decomposition of the representations of simple Lie algebra through characteristic polynomials}\label{CA}
\hspace{1.3em}Let $\mathfrak{g}$ be a finite dimensional complex simple Lie algebra. $\Phi$, $\Pi$, ${\mathfrak{h}}$, $E_{\alpha}$ and $f_\phi(z)$ be as in Section $2$,
and $\phi:\mathfrak{g}\rightarrow \mathfrak{gl}(V)$ be a finitely dimensional  representation of $\mathfrak{g}$. In this section, we first define a new polynomial related to $f_\phi(z)$.
We restrict the  representation $\phi$ of $\mathfrak{g}$ to $\mathfrak{h}$ and obtain a new polynomial in the following definition.
\begin{defn}
 Let $\mathfrak{g}$ and $\phi$ be as  above and
$$\tilde{f}_{\phi}(\tilde z)=\mathrm{det}\left(z_{0}I+\sum_{i=1}^{n}\phi(h_{\alpha_i})z_{i}\right)={f}_{\phi}(z_0,z_1,\cdots,z_n,0,\cdots,0)$$
with $\tilde z=(z_0,z_1,\cdots,z_n)$.
%namely
%$$\tilde{f}_{\phi}(\tilde z),$$
We call $\tilde f_{\phi}(\tilde z)$ the linearization of $f_{\phi}(z)$.
 \end{defn}
For convenience, the polynomials ${f}_{\phi}(z)$ and $\tilde{f}_{\phi}(\tilde z)$ are doneted as ${f}_{\phi}$ and $\tilde{f}_{\phi}$, respectively. We gather these two kinds of polynomials in the two sets,
 $$\mathbf{CP}_{\mathfrak{g}}=\{f_{\phi}(z)\}_{\phi}, \quad \quad \widetilde{\mathbf{CP}_{\mathfrak{g}}}=\{\tilde{f}_{\phi}(\tilde z)\}_{\phi},$$
  where $\phi$ runs over all finitely dimensional  representations of $\mathfrak{g}$.
 %The set consisting of all characteristic polynomials ${f}_{\phi}(z)$ of finite dimensional representations of simple Lie algebra $\mathfrak{g}$ is denoted as $\mathbf{CP}_{\mathfrak{g}}$,
%and the set consisting of all $\tilde{f}_{\phi}(\tilde z)$ is denoted as $\widetilde{\mathbf{CP}_{\mathfrak{g}}}$.
For these two sets, the following proposition  can be proved.
\begin{prop}\label{prop}
		The map $$\rho:\mathbf{CP}_{\mathfrak{g}}\rightarrow\widetilde{\mathbf{CP}_{\mathfrak{g}}},$$
$$\rho({f}_{\phi})=\tilde{f}_{\phi},$$
is a bijective map.
\end{prop}
\begin{proof}
For the representation module $V$ for $\phi$,  we have the complete decomposition of $V=\oplus_{i=1}^{t}V_i$ by Weyl's theorem \cite[Theorem 6.3]{H1972},
where each $V_i$ is an irreducible  module of $\mathfrak{g}$ with highest weight $\beta_i$. Write $\phi=\oplus_{i=1}^{t} \phi_i$, where $\phi_i$ is the representation corresponding to $V_i$. Then we see
\begin{equation} \label{DECOM}
{f}_{\phi}=\prod_{i=1}^{t}{f}_{{\phi}_i}.
\end{equation}
 It is known that each $V_i$ can be written as $V_i=\oplus_{r}\mathbb{C}\mathcal{V}_r$ with $r\in\mathfrak{h}^*=\mathrm{Hom}(\mathfrak{h},\C)$ being a weight of $V_i$,
 such that $h(\mathcal{V}_r)=r(h)\mathcal{V}_r$ for any $h\in\mathfrak{h}$,
  which implies that $\mathfrak{h}$ is diagonalizable on $V_i$.
  We denote all those eigenvectors $\mathcal{V}_{r}s$ for $V_i$ as $\Gamma_i$. It follows that

\begin{equation}\label{linfac}
\tilde{f}_{\phi}=\prod_{i=1}^{t}\prod_{\mathcal{V}_r\in\Gamma_i}\left(z_0+r(h_{\alpha_1})z_{1}+\cdots+r(h_{\alpha_n})z_{n}\right)=\prod_{i=1}^{t}\tilde{f}_{{\phi}_i}.
\end{equation}

For the module $V$, let $\Gamma=\bigcup_{i=1}^{t}\Gamma_i$. Once the set $\Gamma$ is known, we can find one dominant weight  eigenvector $\mathcal{V}_{\beta_0}$ in $\Gamma$ that $\beta_0$ is one of the  highest weights in $\Gamma$.
Therefore the $\mathcal{V}_{\beta_0}$ will generate a unique irreducible representation space of $\mathfrak{g}$, we denote it by $W$. Then $W$ is an irreducible component of $V$. Without loss of generality,  suppose  the module $W$ is  isomorphic to the  module $V_1$.
Thus we can consider $\Gamma'=\Gamma\backslash\Gamma_1$, by induction, we can find one algorithm on $\Gamma$ to determine all the irreducible components of $V$.
Therefore, the polynomial $f_\phi$  can be obtained by formula (\ref{DECOM}) through the algorithm on  $\Gamma$.

Once the polynomial $\tilde{f}_{\phi}$ is fixed, we can uniquely write it as products of linear polynomials with the coefficient of $z_0$ being $1$ as in (\ref{linfac}), and each linear factor can decide a linear function in $\mathfrak{h}^*$,
 because $\{h_{\alpha_i}\}_{i=1}^n$ is a basis of $\mathfrak{h}$.
Then we can decide all the weights for the $\Gamma$ in the above. Therefore, by the algorithm on the set $\Gamma$, the map $\rho$ is an one to one correspondence between $\mathbf{CP}_{\mathfrak{g}}$ and $\widetilde{\mathbf{CP}_{\mathfrak{g}}}$.
\end{proof}

\begin{thm}\label{hinsl2}
Let $\mathfrak{g}$ be a finite dimensional complex simple Lie algebra, $V$ be a $\mathfrak{g}$-module of finite dimension, and $\phi:\mathfrak{g}\rightarrow \mathfrak{gl}(V)$ be a representation of $\mathfrak{g}$.
Then the structure of $V$ as $\mathfrak{g}$-module is uniquely determined by its characteristic polynomial.
\end{thm}

\begin{proof}
Let %$\mathfrak{g}$ be a complex simple Lie algebra,
${\mathfrak{h}}$ be its Cartan subalgebra with dim${\mathfrak{h}}$=$n$, $\Phi$ be the corresponding root system. By the Cartan decomposition of $\mathfrak{g}$, we have
$$\mathfrak{g}={\mathfrak{h}}\oplus \sum_{\alpha\in\Phi} \C E_{\alpha},$$
and let \{$h_{\alpha_1}$,$\cdots$,$h_{\alpha_n}$\} be a basis of ${\mathfrak{h}}$ with $\alpha_i\in\Pi$, $\Pi$ be simple roots of $\mathfrak{g}$.
 Let $f_{\phi}(z)$ be the  characteristic polynomial of $V$ under $\mathfrak{g}$, then

$$f_{\phi}(z)=det\left(z_0I+\sum_{i=1}^{n}\phi(h_{\alpha_i})z_{i}+ \sum_{\alpha\in\Phi} \phi(E_{\alpha})z_{\alpha}\right) .$$
Since we have weight space decomposition $V$=$\oplus_{\lambda\in {\mathfrak{h}}^*}$$V_\lambda$,
%with $\lambda\in {\mathfrak{h}}^* $ which is the dual space of ${\mathfrak{h}}$
and we assume dim$V_\lambda$=$d_\lambda$. We have
$$f_{\phi}(z_0,z_1,\cdots,z_n,0,\cdots,0)=\tilde{f}_{\phi}(z_0,z_1,\cdots,z_n),$$
then $\tilde{f}_{\phi}$ can be rewritten uniquely as product of linear polynomials in the following form
$$\tilde{f}_{\phi}(z_0,z_1,\cdots,z_n)=\prod_{\lambda\in {\mathfrak{h}}^*}\left(z_0+\sum_{\alpha_i\in\Pi}\lambda(h_{\alpha_i})z_i\right)^{d_\lambda}.$$
So by the factorization of $\tilde{f}_{\phi}$, we can clearly know the weights of $\phi$ and its dominant weights relative to $\Pi$.
With the algorithm in the proof of Proposition \ref{prop}, we can determine the structure of $V$ by induction which is similar to a procedure of peeling an onion.
\end{proof}

\section{Monoid structures on $\mathbf{CP}_{\mathfrak{g}}$ and $\widetilde{\mathbf{CP}_{\mathfrak{g}}}$}\label{CA}

 \hspace{1.3em}In this section,we mainly study the algebraic structure of $\mathbf{CP}_{\mathfrak{g}}$ and $\widetilde{\mathbf{CP}_{\mathfrak{g}}}$ inspired by tensor products of the representations of ${\mathfrak{g}}$.

Let ${\mathfrak{g}}$ be a simple Lie algebra with simple roots $\Pi=\{\alpha_i\}_{i=1}^{n}$ and Cartan subalgebra ${\mathfrak{h}}$, having a basis $\{h_{\alpha_i}\}_{i=1}^{n}$.
Let $\mathrm{rep}(\mathfrak{g})$ be the monoidal  category of finitely dimensional representations  of ${\mathfrak{g}}$.
Let $\phi$ be an object in  $\mathrm{rep}(\mathfrak{g})$, $\Gamma_\phi$  be all weights of $\phi$ with  $|\Gamma_\phi|$ being its size.
For each $\lambda\in \Gamma_\phi$, we use $d_{\lambda}$ denote the multiplicity of  $\lambda$ in $\phi$, which is the dimension of the eigenvector space of $\phi$ for $\lambda$.

The set $\widetilde{\mathbf{CP}_{\mathfrak{g}}}$  can be endowed with a monoid structure by a multiplication defined in the below.

\begin{defn}\label{def}
 Let  $\phi$, $\varphi$ be two objects  in  $\mathrm{rep}(\mathfrak{g})$, and $f_\phi$ and $f_\varphi$ be their characteristic polynomials with $\tilde{f}_{\phi}=\rho(f_\phi)$ and $\tilde{f}_{\varphi}=\rho(f_\varphi)$ being their linearizations, respectively.
Suppose that  $\Gamma_\phi=\{\lambda_j\}$,  $\Gamma_{\varphi}=\{\mu_k\}$ are all weights of  $\phi$, $\varphi$, respectively,   and $d_{\lambda_j}$, $d_{\mu_k}$ are their multiplicities in  $\phi$, $\varphi$, respectively.
% be all linear functionals $\lambda_j$ for $\phi$ in the proof of Proposition \ref{prop} with multiplicity $d_{\lambda_j}$ and
Let $$\lambda_{ji}=\lambda_{j}(h_{\alpha_i}),\quad\quad\mu_{ki}=\mu_{k}(h_{\alpha_i})$$
% and $\Gamma_\varphi$ be all linear functionals $\mu_k$ for $\varphi$ with multiplicity $d_{\mu_k}$ and
  for $j=1,\ \cdots, \ |\Gamma_\phi|$, $k=1,\ \cdots,\  |\Gamma_\varphi|$, $i=1,\ \cdots,\  n$.
 It can be seen that
$$\tilde{f}_{\phi}=\prod_{\lambda_j\in\Gamma_\phi }\left(z_0+\sum_{\alpha_i\in\Pi}\lambda_j(h_{\alpha_i})z_i\right)^{d_{\lambda_j}}=\prod_{\lambda_j\in \Gamma_\phi }\left(z_0+\sum_{i=1}^{n}\lambda_{ji}z_i\right)^{d_{\lambda_j}},$$
$$\tilde{f}_{\varphi}=\prod_{\mu_k\in \Gamma_\varphi }\left(z_0+\sum_{\alpha_i\in\Pi}\mu_k(h_{\alpha_i})z_i\right)^{d_{\mu_k}}=\prod_{\mu_k\in \Gamma_\varphi }\left(z_0+\sum_{i=1}^{n}\mu_{ki}z_i\right)^{d_{\mu_k}}.$$
Define $ \tilde{f}_{\phi}*\tilde{f}_{\varphi}\in \mathbb{C}[z_0,z_1,\cdots,z_n]$ for  $\tilde{f}_{\phi}$ and $\tilde{f}_{\varphi}$ by the formula
\begin{equation}
\tilde{f}_{\phi}*\tilde{f}_{\varphi}=\prod_{\lambda_j\in\Gamma_\phi,\mu_k\in \Gamma_\varphi }\left(z_0+\sum_{i=1}^{n} (\lambda_{ji}+\mu_{ki})z_{i}\right)^{d_{\lambda_j} d_{\mu_k}}.
\end{equation}
Here, we call the polynomial $\tilde{f}_{\phi}*\tilde{f}_{\varphi}$ the resolution product of $\tilde{f}_{\phi}$ and $\tilde{f}_{\varphi}$.
 \end{defn}

\begin{prop} \label{tenprod}
  Let $\tilde{f}_{\phi}$, $\tilde{f}_{\varphi}$ be two polynomials in $\widetilde{\mathbf{CP}_{\mathfrak{g}}}$ as in the Definition \ref{def}. It follows that
   $$\tilde{f}_{\phi} * \tilde{f}_{\varphi}=\tilde{f}_{\phi\otimes \varphi}.$$
  \end{prop}
 \begin{proof}
Suppose that $\{v_{\lambda_j}^t\}_{\lambda_j\in\Gamma_\phi}$ with $t=1,\cdots, d_{\lambda_j}$, $\{w_{\mu_k}^s\}_{\mu_k\in\Gamma_\varphi}$ with $s=1,\cdots, d_{\mu_k}$ are bases of representation $\phi$ and $\varphi$, respectively,
such that each $h_{\alpha_i}\in {\mathfrak{h}}$ is diagonalizable under these two bases for $\phi$ and $\varphi$, respectively, which implies that
$$\phi(h_{\alpha_i})(v_{\lambda_j}^t)=\lambda_{ji} v_{\lambda_j}^t, \quad \varphi(h_{\alpha_i})(w_{\mu_k}^s)=\mu_{ki} w_{\mu_k}^s.$$

It is known that $\{v_{\lambda_j}^t\otimes w_{\mu_k}^s\}_{\lambda_j\in\Gamma_\phi,\mu_k\in\Gamma_\varphi}$ is a basis of $\phi\otimes \varphi$ for $t=1,\ \cdots, \ d_{\lambda_j},\  s=1,\cdots, d_{\mu_k}$. Furthermore, we have
    \begin{eqnarray*}
    &&\phi\otimes \varphi(h_{\alpha_i})(v_{\lambda_j}^t\otimes w_{\mu_k}^s)\\
    &=&(\phi(h_{\alpha_i})\otimes I+ I\otimes \varphi(h_{\alpha_i}))(v_{\lambda_j}^t\otimes w_{\mu_k}^s)\\
    &=&\phi(h_{\alpha_i})(v_{\lambda_j}^t)\otimes w_{\mu_k}^s+v_{\lambda_j}^t\otimes \varphi(h_{\alpha_i})(w_{\mu_k}^s)=(\lambda_{ji}+\mu_{ki})(v_{\lambda_j}^t\otimes w_{\mu_k}^s).
    \end{eqnarray*}
    Let
    \begin{equation}\label{for1}
    \phi\otimes \varphi(h_{\alpha_i})_{v_{\lambda_j}^t\otimes w_{\mu_k}^s}=\lambda_{ji}+\mu_{ki}
    \end{equation}
    for  $1\leq j\leq |\Gamma_\phi|$,  $1\leq k\leq |\Gamma_{\varphi}|$, $1\leq i\leq n$, and $d_{\lambda_j\tilde{+}\mu_k}$
    denote the dimension of eigenvector space for $\phi\otimes \varphi(h_{\alpha_i})$  of the eigenvalue $\lambda_{ji}+\mu_{ki}$ spanned by  the vectors $v_{\lambda_j}^t\otimes w_{\mu_k}^s$,
    for  $t=1,\ \cdots, \  d_{\lambda_j}$ and  $s=1,\ \cdots, \ d_{\mu_k}$. Therefore, it follows that
    \begin{equation}\label{for2}
    d_{\lambda_j\tilde{+}\mu_k}=d_{\lambda_j}d_{\mu_k}.
    \end{equation}
   By equations (\ref{for1}) and (\ref{for2}), one  can rewrite the polynomials $\tilde{f}_{\phi}$, $\tilde{f}_{\varphi}$  and $\tilde{f}_{\phi\otimes \varphi}$ as the follows,
 \begin{eqnarray*}
 \tilde{f}_{\phi}&=&\prod_{\lambda_j\in \Gamma_\phi,}\left(z_0+\sum_{i=1}^{n}\lambda_{ji}z_i\right)^{d_{\lambda_j}},\\
 \tilde{f}_{\varphi}&=&\prod_{\mu_k\in\Gamma_\varphi,}\left(z_0+\sum_{i=1}^{n}\mu_{ki}z_i\right)^{d_{\mu_k}},\\
 \tilde{f}_{\phi\otimes \varphi}&=&\prod_{\lambda_j\in\Gamma_\phi, \mu_k\in\Gamma_\varphi }\left(z_0+\sum_{i=1}^{n}\phi\otimes \varphi(h_{\alpha_i})_{v_{\lambda_j}^t\otimes w_{\mu_k}^s}z_i\right)^{d_{\lambda_j\tilde{+}\mu_k}}\\
 &=&\prod_{\lambda_j\in\Gamma_\phi, \mu_k\in \Gamma_\varphi }\left(z_0+\sum_{i=1}^{n}(\lambda_{ji}+\mu_{ki}) z_{i}\right)^{d_{\lambda_j}d_{\mu_k}}= \tilde{f}_{\phi} * \tilde{f}_{\varphi}.
 \end{eqnarray*}
% Therefore, our conclusion holds.
 \end{proof}

 \begin{thm} \label{equalsl2}
 The set $\widetilde{\mathbf{CP}_{\mathfrak{g}}}$ is a  commutative monoid under the resolution product with the unit element $z_0$.
 \end{thm}
 \begin{proof}
  By Proposition \ref{tenprod},  the set $\widetilde{\mathbf{CP}_{\mathfrak{g}}}$ is closed under the resolution product.
   For three representations $\phi$, $\psi$, $\varphi$ of $\mathfrak{g}$, it is known that
 $$\phi\otimes (\psi\otimes \varphi)\simeq (\phi\otimes \psi)\otimes \varphi.$$
 By Proposition \ref{tenprod},
 it follows that
 $$\tilde{f}_{\phi}* (\tilde{f}_\psi*\tilde{f}_\varphi)=\tilde{f}_{\phi\otimes (\psi\otimes \varphi)}=\tilde{f}_{ (\phi\otimes \psi)\otimes \varphi}=(\tilde{f}_{\phi}* \tilde{f}_\psi)*\tilde{f}_\varphi.$$
   Let $\varphi_0$ denote the trivial representation of dimension $1$ of $\mathfrak{g}$, then $\tilde{f}_{\varphi_0}=z_0$.
   For each representation $\phi$  of $\mathfrak{g}$, we have  $\phi\simeq \phi\otimes \varphi_0\simeq \varphi_0\otimes \phi$,  so it follows that
   $\tilde{f}_{\phi}=\tilde{f}_{\phi}*z_0=z_0*\tilde{f}_{\phi}$. At the end,  the relation $\tilde{f}_{\phi}* \tilde{f}_\psi=\tilde{f}_\psi*\tilde{f}_{\phi}$  holds for $\phi\otimes \psi\simeq \psi\otimes \phi$.
 \end{proof}

Let  $\phi$ and $\varphi$ be two representations of $\mathfrak{g}$ of finite dimension.
%and  $\tilde{f}_{\phi\otimes\varphi}$ and ${f}_{\phi\otimes\varphi}$ for  the tensor pruduct $V\otimes W$.
By Proposition \ref{prop},
there is one to one correspondence between  $\mathbf{CP}_{\mathfrak{g}}$ and $\widetilde{\mathbf{CP}_{\mathfrak{g}}}$. Define
$$f_\phi * f_\varphi=\rho^{-1}(\tilde{f}_{\phi}*\tilde{f}_{\varphi}),$$ which is called the resolution product of $f_\phi$ and $f_\varphi$. By Proposition \ref{tenprod}, we have
$$f_\phi * f_\varphi=\rho^{-1}(\tilde{f}_{\phi}*\tilde{f}_{\varphi})=\rho^{-1}(\tilde{f}_{\phi\otimes\varphi})=f_{\phi\otimes\varphi}.$$
 Combining with the Theorem \ref{equalsl2}, the theorem below holds.
 \begin{thm}
 The set $\mathbf{CP}_{\mathfrak{g}}$ is a  commutative monoid under the resolution product with the unit element $z_0$.
Furthermore, the monoids $\mathbf{CP}_{\mathfrak{g}}$ and $\widetilde{\mathbf{CP}_{\mathfrak{g}}}$ are  isomorphic   under their resolution product structures through the linearization map.
 \end{thm}

\section{The adjoint representation of $\mathfrak{sl}(2, \C )$ on simple Lie algebras}\label{CA}

\hspace{1.3em} Let $\mathfrak{L}$ be a simple Lie algebra of finite dimension, $\{ h, e_1, e_2\}$ be a canonical basis of $\mathfrak{sl}(2, \C)$, and $\phi:\mathfrak{sl}(2, \C)\rightarrow \mathfrak{L}$ be a Lie algebra embedding,
 suppose that $\mathrm{ad}\circ \phi$ is the composition of $\phi$ and the adjoint representation $\mathrm{ad}$ of $\mathfrak{L}$. In this section, our main concern is the characteristic polynomial $f_{\mathrm{ad}\circ\phi}$$(z_0,z_1,z_2,z_3)$ of $\mathrm{ad}\circ \phi$.
In \cite{JL2022}, the authors have calculated the case for the simple Lie algebra type $\ddA_{n-1}$. In this section, we're going to calculate this for other types.

   In the following, a general formula in Theorem \ref{homo} will be presented.
   We first apply the formula to the complex simple Lie algebra of type $\ddC_{n}$. Analogous to type $\ddC_{n}$, the results of  other types are summarized  in Table $1$.

  \begin{thm} \label{homo}
Let $\mathfrak{L}$ be as above, $\Phi$ be its root system, and $\lambda \in \Phi$. Let $\phi:\mathfrak{sl}(2, \C)\rightarrow \mathfrak{L}$ be the Lie algebra embedding  defined by $\phi(h)=H_{\lambda}$, $\phi(e_1)=E_{\lambda}$, $\phi(e_2)=E_{-\lambda}$.
 Suppose that $\mathrm{ad}\circ \phi$ is the composition of $\phi$ and the adjoint representation $\mathrm{ad}$ of $\mathfrak{L}$. Then
\begin{equation}\label{FOUR}
f_{\mathrm{ad} \circ\phi}(z_0,z_1,z_2,z_3)=\prod_{i=0}^{3}\left(z_0^2-i^{2}(z_1^2+z_2z_3)\right)^{k_i},
\end{equation}
where $k_i$ represents the multiplicity of the eigenvalue $i$ of  $\mathrm{ad}H_{\lambda}$, for $i$=$0,\ 1,\ 2,\ 3$. Furthermore, we have $dim \mathfrak{L}=k_0+2(k_1+k_2+k_3)$.
 \end{thm}

\begin{proof} From the formula (\ref{FIRST}),  we need to compute the eigenvalues of $\mathrm{ad}H_{\lambda}$ and their multiplicities.
 It is  known that
  $$[H_{\lambda},E_{\beta}]=<\beta,\lambda>E_{\beta}, \ with <\beta,\lambda>=\frac{2(\beta,\lambda)}{(\lambda,\lambda)},$$
for $\beta \in \Phi$, $(\beta,\lambda)$ and $(\lambda,\lambda)$ are the canonical inner products. For $\mathfrak{L}$ is simple Lie algebra, the possible values for $<\beta,\lambda>$ are $0,\  \pm1, \ \pm2,\  \pm3$. Since
$$\mathfrak{L}={\mathfrak{h}}\oplus \sum_{\alpha\in\Phi} \C E_{\alpha}, \quad [H_{\lambda},\mathfrak{h}]=0,$$
for $\mathrm{ad}H_{\lambda}$,  we have all its possible eigenvalues being $0,\  \pm1, \ \pm2,\  \pm3$. Therefore, the formula (\ref{FOUR}) follows from formula (\ref{FIRST}).
 \end{proof}

\begin{rem} \label{LEIN}
Let $\mathfrak{L}$, $\Phi$ be as above, since for each root $\alpha, \beta\in\Phi$, if $\alpha, \beta$ have the same length, there is an element $g$ of Weyl group of the same type,
such that $g\alpha=\beta$, which can be extended to an automorphism of $\mathfrak{L}$. By Corollary \ref{invar}, we see that the characteristic polynomial is invariant under $g$. Then for Theorem \ref{homo},
we can just consider $\lambda$ being one simple root which is a long or a short root.
\end{rem}

\begin{rem} \label{basis}
 Let $\{\epsilon_i\}_{i=1}^{n}$ be the canonical basis of $\R^n$.  From \cite{H1972}, we know  $\Phi$ of type $\ddC_{n}$ can be realized in $\R^n$,
  which consists of $2n$ long roots $\pm2\epsilon_i$ $(1\leq i \leq n)$ and $2n^2-2n$ short roots $\pm\epsilon_i\pm\epsilon_j$ $(1\leq i<j \leq n)$,
   with simple roots $\Pi=\{\alpha_1,\cdots,\alpha_n\}=\{\epsilon_1-\epsilon_2,\cdots,\epsilon_{n-1}-\epsilon_n,2\epsilon_n\}$.
\end{rem}

 We compute the case of $\mathfrak{L}$ being type $\ddC_{n}$ as an example in the following theorem.

\begin{thm} \label{2ton}
 Let $\phi:\mathfrak{sl}(2, \C)\rightarrow \mathfrak{L}$ be the Lie algebra embedding  defined by $\phi(h)=H_{\lambda}$, $\phi(e_1)=E_{\lambda}$, $\phi(e_2)=E_{-\lambda}$ with $\lambda\in\Phi$, where $\mathfrak{L}$ is the complex simple Lie algebra of type $\ddC_{n}$.
 Suppose that $\mathrm{ad}\circ \phi$ be the composition of $\phi$ and the adjoint representation $\mathrm{ad}$ of $\mathfrak{L}$. Then
\begin{equation}
f_{\mathrm{ad} \circ\phi}=z_0^{2n^2-3n+6}\left(z_0^2-(z_1^2+z_2z_3)\right)^{2n-4}\left(z_0^2-4(z_1^2+z_2z_3)\right)
\end{equation}
with $\lambda$ being a long root.
\begin{equation}
f_{\mathrm{ad} \circ\phi}=z_0^{2n^2-7n+10}\left(z_0^2-(z_1^2+z_2z_3)\right)^{4n-8}\left(z_0^2-4(z_1^2+z_2z_3)\right)^3
\end{equation}
with $\lambda$ being a short root.

 \end{thm}

\begin{proof}
By Remark \ref{LEIN}, without loss of generality, suppose that $\lambda$=$2\epsilon_n$ for the long root case.
By Theorem \ref{homo}, it is natural to  calculate the multiplicity of each eigenvalue, whose possible values are $0, \pm 1, \pm 2$.
Through $\frac{2(\beta,2\epsilon_n)}{(2\epsilon_n,2\epsilon_n)}$, we can find the number of $E_\beta$ for the eigenvalue $<\beta,\alpha>$ by realizing $\Phi$ in the Remark \ref{basis}.
Hence, it follows that $k_0=2n^2-4n+6$, $k_1=2n-4$, and $k_2=1$.

In a similar way,  we assume $\lambda$=$\epsilon_1-\epsilon_2$ for the short root case. By calculating $\frac{2(\beta,\epsilon_1-\epsilon_2)}{(\epsilon_1-\epsilon_2,\epsilon_1-\epsilon_2)}$,
it can be seen that  $k_0=2n^2-8n+10$, $k_1=4n-8$, and $k_2=3$. So the formulas yields from Theorem \ref{homo}.
 \end{proof}
We list the crucial values $k_0,$ $k_1,$ $k_2,$ $k_3$ in Table $1$ for all simple types. In the table, the letters $\alpha$ and $\gamma$
 stand for a long root and a short root, respectively.
\begin{table}[htbp]
\caption{$The\ values\ of\ k_0,\ k_1,\ k_2,\ k_3$}
\centering
\renewcommand\arraystretch{1.5}{
\begin{tabular}{|p{1cm}<{\centering}|p{1cm}<{\centering}|p{2.35cm}<{\centering}|p{2.3cm}<{\centering}|p{2.3cm}<{\centering}|p{2.3cm}<{\centering}|}
\hline
 & & $k_0$ & $k_1$ & $k_2$ & $k_3$\\
\hline
$Type$ & $Root$ & $<\beta,\alpha>$=0 & $<\beta,\alpha>$=1& $<\beta,\alpha>$=2& $<\beta,\alpha>$=3 \\
\hline
$A_n$ & $\alpha$ & $n^2-5n+6$ & $2n-4$ & $1$ & $0$ \\
\hline
\multirow{2}{*}{$B_n$} & $\alpha$ & $2n^2-7n+14$ & $4n-8$ & $1$ & $0$\\
\cline{2-6}
\multicolumn{1}{|c|}{}& $\gamma$ & $2n^2-3n-2$ & 0 & $2n+1$ & 0\\
\hline
\multirow{2}{*}{$C_n$} & $\alpha$ & $2n^2-3n+6$ & $2n-4$ & $1$ & $0$ \\
\cline{2-6}
\multicolumn{1}{|c|}{}& $\gamma$ & $2n^2-7n+10$ & $4n-8$ & $3$ & $0$\\
\hline
$D_n$ & $\alpha$ & $2n^2-9n+14$ & $4n-8$ & $1$ & $0$ \\
\hline
\multirow{2}{*}{$G_2$} & $\alpha$ & 2 & 4 & 1 & 0\\
\cline{2-6}
\multicolumn{1}{|c|}{}& $\gamma$ & 2 & 2 & 1 & 2\\
\hline
\multirow{2}{*}{$F_4$} & $\alpha$ & 18 & 14 & 1 & 0\\
\cline{2-6}
\multicolumn{1}{|c|}{}& $\gamma       $ & 18 & 14 & 1 & 0\\
\hline
$E_6$ & $\alpha$ & 30 & 12 & 9 & 0  \\
\hline
$E_7$ & $\alpha$ & 60 & 16 & 17 & 0 \\
\hline
$E_8$ & $\alpha$ & 126 & 24 & 33 & 0 \\
\hline
\end{tabular}}
 %\footnotesize{ $\alpha \ stands\ for\ long\ root,\ \gamma \ stands\ for\ short\ root$}\\
\end{table}

\newpage

\section{Borel subalgebras, parabolic subalgebras and spectral matrix}

\hspace{1.3em}Let's first recall the spectral matrix of a solvable Lie algebra from \cite{KY2021} in the following.
 \begin{defn}\label{respro}
 Let $\mathfrak{L}$ be a solvable Lie algebra with a basis  $\mathcal{B}=\{\upsilon_1,\cdots,\upsilon_n\}$
and $f_{\mathfrak{L}}(z)$ be its characteristic polynomial for the adjoint representation. By Hu-Zhang's theorem $\cite{HZ2019}$, one may write
$$f_{\mathfrak{L}}(z)=\prod_{j=1}^{n}\left(z_0+\sum_{i=1}^{n}\lambda_{ij}z_i\right) .$$
 The spectral matrix for
$\mathfrak{L}$ with respect to the basis $\mathcal{B}$ is defined as  $\lambda_{\mathfrak{L}}$=$(\lambda_{ij})_{n\times n}$, namely
\[ \lambda_{\mathfrak{L}} =
   \begin{pmatrix}
   \lambda_{11}& \cdots &\lambda_{1n} \\ \vdots&\ddots&\vdots  \\  \lambda_{n1} &\cdots&\lambda_{nn} \\
     \end{pmatrix}.
   \]
\end{defn}
 In this section, we focus on the rank of spectral matrices and parabolic subalgebras.
 % Let $\mathfrak{g}$ be a simple Lie Algebra, containing a Borel subalgebra $B$, and $\mathfrak{h}$ is a Cartan subalgebra of $\mathfrak{g}$.
 By computing the characteristic polynomial of the Borel subalgebra of a simple Lie algebra,  the theorem below can follow.

 \begin{thm} \label{2ton}
Suppose $\mathfrak{g}$ is a simple Lie Algebra, with $\mathfrak{h}$ being a Cartan subalgebra of $\mathfrak{g}$, and $\Phi^+$ being the positive root system of $\mathfrak{g}$. Suppose
$$B=\mathfrak{h}\oplus\sum_{\alpha\in\Phi^{+}} \C E_{\alpha},$$
 a Borel subalgebra of $\mathfrak{g}$. We have $\mathrm{rank}\lambda_{B} = \mathrm{dim}\mathfrak{h}$.
 \end{thm}
\begin{proof}
 Take  a basis of $(h_1,\cdots, h_n)$ for $\mathfrak{h}$, and suppose that $\mathrm{ht}(\alpha)$ is the height of $\alpha\in \Phi^{+}$ relative to some simple roots. Arrange the positive roots in $\Phi^+$ by their heights such that $\Phi^+=\{\alpha_1,\cdots,\alpha_s\}$ with $s\geq j\geq i\geq 1$ indicating $\mathrm{ht}(\alpha_j)\geq \mathrm{ht}(\alpha_i)$ and their corresponding root vectors being $E_{\alpha_1},\ \cdots,\ E_{\alpha_s}$. Since
\begin{equation*}
\begin{split}
 &[\mathfrak{h},\mathfrak{h}]=0,\\
 &[h_i,E_{\alpha_j}]=\alpha_{j}(h_i)E_{\alpha_j}, \forall 1\leq i \leq n,1\leq j\leq s,\\
 &[E_{\alpha_i},E_{\alpha_j}]=N_{ij}E_{\alpha_i+\alpha_j}, 1\leq i,j \leq s,\alpha_i+\alpha_j \in \Phi^{+}, \\
 &[E_{\alpha_i},E_{\alpha_j}]=0, 1\leq i,j \leq s,\alpha_i+\alpha_j \notin \Phi^{+}.
 \end{split}
        \end{equation*}
Then we see $\mathrm{ht}(\alpha_i+\alpha_j) > \mathrm{ht}(\alpha_i)$ and $\mathrm{ht}(\alpha_i+\alpha_j) > \mathrm{ht}(\alpha_j)$ if $\alpha_i+\alpha_j \in \Phi^{+}$.
Under the ordered basis $\mathcal{B}=\{h_1,\cdots, h_n,E_{\alpha_1},\cdots,E_{\alpha_s} \}$,  it follows that
  \begin{equation*}
            \begin{split}
                &{\rm ad}h_i=\begin{pmatrix}
                    0 & 0\\
                    0 & A_i
                \end{pmatrix}
             \quad with \quad
                A_i=\begin{pmatrix}
                    \alpha_1(h_i) & \cdots & 0\\
                    \vdots & \ddots & \vdots\\
                    0 & \cdots & \alpha_s(h_i)
                \end{pmatrix},\\
                &{\rm ad}E_{\alpha_j}=\begin{pmatrix}
                    0I_{n\times n} & &\\
                    \vdots & \ddots &\\
                    * & \cdots & 0
                \end{pmatrix},
            \end{split}
        \end{equation*}
   such that $A_i$ is a diagonal matrix and $adE_{\alpha_j}$ is a strict lower triangular matrix.
Then the characteristic polynomial for the adjoint representation with the  ordered basis $\{h_1,\cdots,h_n,E_{\alpha_1},\cdots,E_{\alpha_s}\}$ of $B$ is
 \begin{equation*}
            f_B(z)=\begin{vmatrix}
                z_0I_{n\times n} & 0\\
                * & C
            \end{vmatrix}
   \end{equation*}
   with \begin{equation*}
              C=\begin{pmatrix}
                z_0+\sum_{i=1}^{n}\alpha_1(h_i)z_i & \cdots & 0\\
                \vdots & \ddots & \vdots\\
                0 & \cdots & z_0+\sum_{i=1}^{n}\alpha_s(h_i)z_i
            \end{pmatrix},
 \end{equation*}
which implies
 \begin{equation*}
            f_B(z)=z_0^n\prod_{j=1}^s\left(z_0+\sum_{i=1}^{n}\alpha_j(h_i)z_i\right).
        \end{equation*}
Therefore, we have
 \begin{equation*}
            {\rm rank}\lambda_B={\rm rank}\begin{pmatrix}
                0I_{n\times n}&    \\
                 &\alpha_1(h_1) & \cdots & \alpha_1(h_n) \\
                 &\vdots & \ddots & \vdots\\
                 &\alpha_s(h_1) & \cdots & \alpha_s(h_n)
            \end{pmatrix}={\rm dim}\mathfrak{h}.
        \end{equation*}
   \end{proof}

\begin{rem}
It also can be obtained by $Rank\lambda_B$=dim $B/Nil(B)$ (\cite[proposition 4.5]{KY2021}) through proving $Nil(B)$=$\sum_{\alpha\in\Phi^{+}} \C E_{\alpha}$, which can be done in analyzing its structure. For more information about the nilpotent radical $Nil(B)$, one can find it in
\cite[Chapter 5]{K2002}.
\end{rem}

\begin{thm}
Let $\mathfrak{p}$ be a parabolic subalgebra of $\mathfrak{g}$ a complex simple Lie algebra.
% so $\mathfrak{p}$ contains a maximal solvable subalgebra(a Borel subalgebra) of $\mathfrak{g}$.
Then $\mathfrak{p}$ is a Borel subalgebra if and only if its characteristic polynomial of any linear representation of finite dimension is a product of linear factors.
\end{thm}

\begin{proof}
If $\mathfrak{p}$ is a Borel subalgebra of $\mathfrak{g}$, then $\mathfrak{p}$ is solvable. By Lie's Theorem, there exists a basis of the complex linear space $V$ such that the matrix of $\mathfrak{p}$ is upper
triangular relative to the basis. Therefore the necessity holds.
On the other hand, recall the Hu-Zhang's theorem \cite[Theorem 5.4]{HZ2019} which says that a finite dimensional Lie algebra is solvable if and only if its characteristic polynomial with respect to any linear representation is a product of linear factors.
By the assumption, if the characteristic polynomial of a linear representation of $\mathfrak{p}$ is a product of linear factors, it implies that $\mathfrak{p}$ is solvable.
\end{proof}

\begin{rem} From this paper, we see that the characteristic polynomials of the representations of a  complex simple Lie algebra have a profound meaning for the representation monoidal category of the Lie algebra.
There are many interesting topics about these polynomials, such as how to present these polynomials precisely, finding the link between the coefficients  and the representations,
 how to factorize the tensor products through the resolution products of their characteristic polynomials and so on. Therefore, we need more efforts to work on these topics.
\end{rem}

Amin Geng\\
Email: gengamin@163.com\\
School of Mathematics, Shandong University\\
Shanda Nanlu 27, Jinan, \\
Shandong Province, China\\
Postcode: 250100\\
Shoumin Liu\\
Email: s.liu@sdu.edu.cn\\
School of Mathematics, Shandong University\\
Shanda Nanlu 27, Jinan, \\
Shandong Province, China\\
Postcode: 250100\\
Xumin Wang\\
Email: 202011917@mail.sdu.edu.cn\\
School of Mathematics, Shandong University\\
Shanda Nanlu 27, Jinan, \\
Shandong Province, China\\
Postcode: 250100

\end{document}